\documentclass[12pt]{article}%
\usepackage{amsfonts}
\usepackage{amsmath}
\usepackage{amssymb}
\usepackage{amsthm}
\usepackage{amscd}
\usepackage{graphicx,epsf}
\usepackage{pgfplots}
\usepackage{enumerate}
\usepackage{csquotes}
\usepackage{url}
\usepackage{xcolor}
\usepackage{graphicx}%
\providecommand{\U}[1]{\protect\rule{.1in}{.1in}}

\newtheorem{cor}{Corollary}
\newtheorem{prop}{Proposition}
\newtheorem{lem}{Lemma}
\newtheorem{con}{Conjecture}
\theoremstyle{definition}
\newtheorem{rem}{Remark}
\newtheorem{exa}{Example}

\usepackage{xcolor}
\usepackage{ifthen}
 \newboolean{pedro}
\setboolean{pedro}{true}
\newcommand{\pedro}{}
 \newcommand{\frakp}{\mathbf{P}}
\newcommand{\ttx}{\theta_{2}}
 
\begin{document}

\title{\textbf{Nash Equilibria in the Showcase Showdown game with unlimited spins}}
\author{L. Bay\'on\\Departmento de Matem\'aticas, Universidad de Oviedo\\Avda. Calvo Sotelo, s/n. 33007 Oviedo (Spain)\\bayon@uniovi.es\\P. Fortuy Ayuso\\Departmento de Matem\'aticas, Universidad de Oviedo\\Avda. Calvo Sotelo, s/n. 33007 Oviedo (Spain)\\fortunypedro@uniovi.es\\J.M. Grau \\Departmento de Matem\'aticas, Universidad de Oviedo\\Avda. Calvo Sotelo, s/n. 33007 Oviedo (Spain)\\grau@uniovi.es\thanks{Corresponding author}\\A.M. Oller-Marc\'en\\Departamento de Matem\'aticas - IUMA, Universidad de Zaragoza\\Pedro Cerbuna, 12. 50006 Zaragoza (Spain)\\oller@unizar.es\\M.M. Ruiz\\Departmento de Matem\'aticas, Universidad de Oviedo\\Avda. Calvo Sotelo, s/n. 33007 Oviedo (Spain)\\mruiz@uniovi.es }
\maketitle

\begin{abstract}

The game of \emph{Showcase Showdown}
with unlimited spins is investigated as an $n$-players continuous game, and the
Nash Equilibrium strategies for the players are obtained. The sequential game
with information on the results of the previous players is studied, as well as
three variants: no information, possibility of draw, and different modalities
of winner payoff.

\end{abstract}

\bigskip

\bigskip

\renewcommand{\baselinestretch}{1}


\section{Introduction}

  A well known example of applied probability in TV is the game
show: ``The Price is Right''. One of the games of the show is called ``The
Showcase Showdown.'' In it, each of three players spins a wheel in turns. The
wheel has 20 values, between 5c and \$1, in increments of $5$c. Each player
spins once, and then, after seeing the result, has two options: either spinning
again, or  stopping. There are a maximum of two spins and if the player
spins twice, both values are added. If a player exceeds a dollar, he is
immediately eliminated. Otherwise, the turn passes to the next player. The aim
of each player is to obtain the highest score less than or equal to $1$. Each
player knows the results of the previous ones. The natural question is: What is
the best strategy for playing Showcase Showdown?

  Coe and Butterworth \cite{Coe} define the optimal stopping time for player $i$
$(i=1,2,3)$ as the smallest value such that it is better for player $i$ to
stop after the first spin rather than spinning a second time.
Tenorio and Cason \cite{Tenorio2002} also analyzed this discrete game with three players,
assuming that the result of each spin is a discrete random variable, uniformly
distributed on the set $\{0.05,0.10,...,1.00\}$, and that the spins are
independent random events.

Kaynar \cite{Kaynar} considers a variant in which each player draws one
or two random numbers between $0$ and $1$, and where each player has no
information about the results and actions of the previous ones. The optimal
solution with two possible attempts and two and three players is also given
there.

 In several studies, Sakaguchi \cite{Sakaguchi2004two-player,
Sakaguchi2004equilibrium} covers the game for two players, each of whom can
play once or twice, and assuming each sampling follows a uniformly distributed
random variable in $[0,1]$. The aim of the version of ``Showcase Showdown'' he
calls (GSS) is to obtain the highest score among all of the players in the
game, from one or two chances of sampling. He also studies three different
versions of this game, depending on the scoring function: ``Keep-or-Exchange''
(GKE), ``Competing Average'' (GCA), and ``Risky Exchange'' (GRE).
In \cite{Sakaguchi2004two-player}, he solves GSS, while
in \cite{Sakaguchi2004equilibrium} he does so for GKE and GRE. The version GCA has
not been solved yet. Also, in \cite{Sakaguchi2005threeplayerRE, Sakaguchi2005threeplayerKE} the same author studies the games with two sampling possibilities
but \emph{three players}.

 Elsewhere, also Sakaguchi \cite{Sakaguchy2007}, solves the GKE and GRE versions
with two players and three different types of information: no information
sharing at all, that each player informs the other one of his results, and
that the first one informs the second one of his results but not the other way
around.

  Another generalization is analyzed by Swenson \cite{Swenson2015}, with the
same rules as Showcase Showdown with $n$ players, but where each spin follows a
continuous random variable uniformly distributed on $[0,1]$: the $n$-player
continuous game. He raises the question of the existence of optimal cutoff
values for the case in which more than two spins are allowed.

  Our aim is to study the generalization of Showcase Showdown to
any number of samplings (called \emph{spins}) and of players. This has been
scarcely studied. Mazalov and Ivashko \cite{Mazalov} analyze the GSS game with no information
(each player makes his decisions with no knowledge of the results of the
others). Using dynamic programming theory, they find the Nash equilibrium of
the $n$-player GSS with an infinite number of spins: a state when it is
unprofitable for all players to individually deviate from their strategies,
assuming that if the scores of all players are more than $1$, then the winner
is the one whose score is closest to $1$. Seregina, Ivashko and Mazalov \cite{Seregrina} study the same
no-information version but with $n$ spins, providing the optimal payoffs. The
optimal strategies of the players in the version with complete information are
studied and the optimal payoffs for the first player are computed.

This work covers the $n$-player GSS with an infinite number of spins, where the aim of
each player is to obtain the highest total score less than (or equal to) $1$, allowing for
the possibility of draw among players when all their scores are strictly greater than
$1$. We study the sequential version of the game, with information on the previous players
results, studying their optimal strategies and obtaining formulas which allow us to
compute the optimal payoffs not only of the first player but of all of them. We also state
and solve three variants of the no-information version with different payoff modes under
the condition if all the scores are above $1$, then the payoff is $0$. Finally, we study
the Nash equilibria as the expected gains in each variant.

\pedro{}
In summary, our main contributions are:
\begin{itemize}
\item[(A)] We provide formulas allowing the computation of the winning probability of each player in the sequential game, when all of them play maximizing their own winning probability (Proposition \ref{estrategia}).
\item[(B)] We describe and provide examples of how, in the sequential case, coalitions between players can reduce other player's winning probabilities (Section 5.1).
\item[(C)] We compute the formulas giving the Nash equilibria in different versions of the game; these formulas show that, despite the essential aim of the game being the same, the diverse payoffs can give rise to great changes in the equilibrium strategies (Propositions \ref{pro:equilibrium-game-2-1}, \ref{pro:equilibrium-game-2-2}, and \ref{pro:gameII-3}, one for each variant of the game we study).
\end{itemize}
\pedro{}

  The paper is organized as follows: In Section 2, we describe the
  versions of the $n$-person Showcase Showdown game with unlimited spins we cover. In section 3 we present a version \emph{ad hoc} of the one-stage look-ahead (OLA) stopping rule. In section 4 we compute the cumulative distribution function of the random variable given by the score of each player depending on his greed threshold. 
  The optimal strategies and expected payoffs of all players are given for the sequential game in Section 5. In Section 6, we study three variants of the no-information game, with different payoffs for the winner. Finally, in section 7, some prospects for the future are presented that we consider interesting.

\section{The Showcase Showdown game}

We shall consider several cases of the following version of the Showcase Showdown
game, with set of players $\{A_{i}\}_{i=1}^{n}$: each player $A_{i}$ starts the game with
a value $S(A_i)=0$. Then he successively retrieves a value from a uniformly distributed
random variable in $[0,1]$, and adds this value to $S(A_i)$ (this act will be called a
\emph{play} of the player); this retrieval is repeated until either he stops or
$S(A_i)>1$. If he stops with $S(A_{i})\leq 1$, then his score is $S(A_{i})$, otherwise
$S(A_{i})$ is irrevocably set to $S(A_i)=0$, and player $A_{i+1}$ starts his turn. The
winner is the player with the greatest score, and there is a draw if all the scores are
$0$.

The variants we shall consider are the following, in all of which the final payoff for the winner is $1$.

\begin{enumerate}
\item  \textsc{Game i} --- \emph{Sequential game}. Player $A_{i}$
knows the values $S(A_{j})$ for $j=1,\ldots, i-1$.

\item \textsc{Game ii} --- \emph{No-information game}. There is no information available
  about $S(A_{j})$ \pedro{}for any other player\pedro{}. There are three sub-variants: the first two
  depending on where the payment comes from (either an external agent or the rest of the
  players), and the third one, in which \pedro{}one pre-selected player, say $A_n$\pedro{}, is the winner in case of draw.

\begin{itemize}
\item \textsc{Game ii.1} --- \emph{Non-constant-sum}. The \pedro{} winner's payoff is provided\pedro{} by an external agent. If
  $S(A_{i})=0$ for all $i$, then the payoff is $0$ for all players. \pedro{}Thus, the sum of all payoffs may be $0$.\pedro{}

\item \textsc{Game ii.2} --- \emph{Zero-sum}. The payoff of the winner is collected from
  the other players, each providing $1/(n-1)$. As in the previous case, the payoff is $0$
  if $S(A_{i})=0$ for all $i$.

\item \textsc{Game ii.3} --- \emph{Non-symmetric and constant-sum}. There is a known  player $A_{j}$ with advantage: if $S(A_{i})=0$ for all $i$, then $A_{j}$ wins. We shall
  set $j=n$ as the index is irrelevant.
\end{itemize}
\end{enumerate}

We shall show how the reasonable strategies are all based on establishing a \emph{greed   threshold} for each player $A_i$: a value \pedro{}$\kappa_i\in(0,1)$\pedro{} such that $A_i$ continues
playing (i.e. retrieving a random number and adding it to $S(A_i)$) until $S(A_i)>\kappa_i$. This way, the no-information variants can be understood as continuous games in which each player $A_{i}$
computes his greed threshold $\kappa_i$, and where the final payoff is given by some
functions $\mathbb{P}_i(\kappa_1,\ldots, \kappa_n)$. These payoffs depend, obviously, on
the specific game. In the sequential version (\textsc{Game i}), we shall also see how the
optimal policy is of this type and depends on the score of the previous players,
and the number of players still to play.

\section{Optimal strategy. Threshold strategy.}

In this section we restrict ourselves to the Showcase Showdown game with a single player with payoff $h(x)$ if he stops with score $S=x$. The function $h(x)$ is defined in $[0,1]$
and assumed non-decreasing. This includes, for instance, the payoff $h(x)=x$ (ordinary single-player Showcase Showdown). This situation can also happen with several players in
the sequential game \textsc{ (Game I)} and where $h(x)$ is the probability that no later player gets a score greater than $x$. In no-information games, $h(x)$ will represent the expected payoff when stopping at $x$, assuming certain rival strategies.

Roughly speaking, one can say that the optimal policy in this case (single player)
consists in: given the value $S$, decide which of playing again or stopping has better
expected payoff, and act accordingly. Let $G (x)$ be the expected payoff resulting from
playing again and following the optimal policy (whatever this may be) from that point
on. When $S=x$, the expected payoff following that strategy is necessarily
$\max(h(x),G (x))$, so that the expected payoff $G (x)$ of continuing playing with
$S=x$ and following, later on, the optimal policy, must satisfy:

\begin{equation*}
G (x)= h(0)x +\int_0^{1-x}  \max(h(x+t),G (x+t)) dt,
\end{equation*}
where $h(0)$ is the payoff when $S>1$ (i.e. final score $0$); the probability of this
when $S=x$ is $x$. This gives the following integral equation for $G (x)$:
\begin{equation*}
G (x)= h(0) x + \int_x^1 \max(h(t),G (t))dt.
\end{equation*}
The following lemma proves that $G (x)$ exists and can be explicitly defined (so that it
is also unique) in terms of $h(x)$. This is an \emph{ad-hoc} version of the one-stage
look-ahead (OLA) stopping rule (\cite{Mazalov}), which compares the payoff if
stopping when $S=x$ with the expected payoff of making a single play more, and
stopping.

\pedro{}
\begin{lem}\label{lem:psi-max}
Let $h:[0,1]\rightarrow\mathbb{R} $ be a non-decreasing monotone function, and let $\tilde{h}(x)=h(0) x + \int_x^1h(t)dt$, for $x\in[0,1]$. Define
\begin{equation*}
  \label{eq:kappa}
  \kappa := \inf \left\{x \in[0,1]: h(x)\geq  \tilde{h}(x) \right\}.
\end{equation*}
Then the function
\begin{equation}
  \label{eq:Psi}
  G (x):=
  \begin{cases}
    (\tilde{h}(\kappa)-h(0)) e^{\kappa -x}+h(0) & \text{if $x<\kappa$},
    \\
        \tilde{h}(x) & \text{if $x\geq\kappa$ }.
  \end{cases}
\end{equation}
is a non-increasing monotone  $C[0,1]$ function, and the only one
satisfying the integral equation
\begin{equation}
  \label{eq:Psifunc}
  y (x)= h(0) x  + \int_x^1 \max(h(t),y (t))dt.
\end{equation}
\end{lem}
\begin{proof}
  Before proceeding, notice that $h(x)$ being non-decreasing implies that it is Riemann integrable in $[0,1]$, so that the definition of $\tilde{h}(x)$ and Equation \eqref{eq:Psifunc} make sense. By definition, $\tilde{h}(x)$ is a continuous function, so that $\tilde{h}(\kappa)$ is well-defined.

  Let us first  verify that $G (x)$ as defined in \eqref{eq:Psi} is a solution of \eqref{eq:Psifunc}. By construction, $G (x)$ is continuous, as both parts are continuous and they coincide at $x=\kappa$. Also by construction, it is non-increasing for $x<\kappa$. Now, if $x\geq \kappa$ and $\epsilon \geq 0$, then, by definition of $\kappa$ and $G (x)$, and because $h(x)$ in non-decreasing, we have:
  \begin{multline*}
    \label{eq:psi-decr}
    G (x+\epsilon) = \tilde{h}(x) = h(0)(x+\epsilon) + \int_{x+\epsilon}^1h(t)dt = h(0)x + h(0)\epsilon +\int_{x+\epsilon}^1h(t)dt \leq\\
    h(0)x + h(x)\epsilon + \int_{x+\epsilon}^1h(t)dt \leq h(0)x +
    \int_{x}^{x+\epsilon} h(t)dt + \int_{x+\epsilon}^1h(t)dt = \\
    h(0)x + \int_x^1 h(t)dt = \tilde{h}(x) = G (x),
  \end{multline*}
so that $G (x)$ is non-increasing for $x\geq \kappa$. Thus, $G(x)$ is non-increasing in the whole interval $[0,1]$.

By construction, and by definition of $\kappa$, $G (x)=\tilde{h}(x)\leq h(x)$ for $x> \kappa$. Also, as $h(x)$ is non-decreasing and $G (x)$ is non-increasing, we infer that $G (x)\geq h(x)$ for $x<\kappa$. Thus, $\kappa$ satisfies also:
\begin{equation*}
  \label{eq:kappa-psi-h}
  \kappa = \inf \left\{ x\in[0,1] : G (x) \leq h(x) \right\},
\end{equation*}
 which implies that for $x\in (\kappa,1]$, we have $\max(h(x),G(x))=h(x)$, which gives:
 \begin{equation*}
   \label{eq:Psi-cumple-eq-2}
   G (x)=h(0)x  +
   \int_x^1
   \max(h(t),G (t))dt\;
   \mathrm{for}\; x>\kappa,
 \end{equation*}
 that is: $G (x)$ satisfies \eqref{eq:Psifunc} for $x > \kappa$ and obviously too for $x=\kappa$.
 
 On the other hand, as $G(x)$ is non-increasing, we obtain $G(x)\geq h(x)$ for $x<\kappa$. This gives, for $x\in [0,\kappa)$:
 \begin{equation*}
   \label{eq:Psi-izda1}
   G(x) = h(0)x + \int_x^\kappa G(t)dt + \int_{\kappa}^1h(t)dt,
 \end{equation*}
 which is, by construction,
 \begin{equation*}
   \label{eq:Psi-total}
   G(x) = h(0)x + \int_x^1\max(h(t),G(t))dt,
 \end{equation*}
 as required. Thus, $y(x)=G(x)$ is a solution of \eqref{eq:Psifunc}.

  In order to show its uniqueness, we apply Banach's fixed point theorem. Let $\alpha\in(0,1]$ and consider the map:
  \begin{equation*}
    \begin{array}{rcl}
      C[\alpha,1] & \stackrel{\varphi}{\longrightarrow} & C[\alpha, 1]\\
      y(x) & \longmapsto & h(0)x + \int_x^{1}\max(h(t),y(t))dt
    \end{array}
  \end{equation*}
  and, in $C[\alpha,1]$ consider the supremum metric. Take $y_0,y_1\in C[\alpha, 1]$. We have
  \begin{equation*}
    \label{eq:supremum}
    \|y_0 - y_1\| = \max_{x\in [\alpha,1]}\left|h(0)x + \int_x^{1}
      \max(y_0(t),h(t))dt - h(0)x -
      \int_{x}^1\max(y_{1}(t),h(t))dt\right|,
  \end{equation*}
  so that
  \begin{equation*}
    \label{eq:sup2}
    \|y_0-y_1\| \leq \max_{x\in [\alpha,1]}
      \int_x^{1}|\max(y_0(t),h(t)) - \max(y_1(t),h(t))|dt.
  \end{equation*}
  Given three real numbers $a,b,c$, $|\max(a,b)-\max(a,c)|\leq \max(b-c)$: the only possible values of $|\max(a,b)-\max(a,c)|$ are: $0$, $|b-c|$, $|a-c|$ and $|a-b|$, but the last two can only happen if $a$ is between $b,c$, so that in any case $|\max(a,b)-\max(a,c)|\leq |b-c|$. As a consequence,
  \begin{equation*}
    \label{eq:sup3}
    \|y_0-y_1\| \leq \max_{x\in[\alpha,1]}\int_x^1|y_0(t)-y_1(t)|dt \leq
    (1-\alpha)\max_{t\in[x,1]}|y_0(t)-y_1(t)|\leq (1-\alpha)\|y_0-y_1\|.
  \end{equation*}
  As $\alpha>0$, we deduce that $\varphi$ is a contraction map, and has a single fixed point in $C[\alpha,1]$ for any $\alpha$. As $G $ is continuous in $C[0,1]$, its restriction to any $C[\alpha,1]$ is that unique solution. This gives the uniqueness of $G $, as $G (0)$ is determined by its continuity.
\end{proof}
\pedro{}
In the following result, ``optimal'' means ``the expected payoff is maximum''. Notice that both ``stopping'' and ``continuing playing'' might be optimal at the same time.

\begin{prop} \label{OLA} Consider the single-player Showcase Showdown game with payoff
  function $h(x)$ for score $x$. Assume $h:[0,1]\rightarrow\mathbb{R}$ is non-decreasing
  monotone and  let $\tilde{h}(x)=h(0) x + \int_x^1h(t)dt$, for $x\in[0,1]$. Then, there exists an optimal policy \pedro{}which\pedro{} is of threshold
  type. In other words, there is $\kappa\in [0,1]$ (optimal threshold) such that if at some point
  the total score is $S< \kappa$, then the optimal decision consists in continuing
  playing, whereas if $S> \kappa$ then the optimal decision is to stop. Finally, if $S=\kappa$, the optimal decision is to stop if and only if $h(\kappa)\geq \tilde{h}(\kappa) $. Furthermore,
  \begin{equation*}
    \kappa = \inf \left\{x \in[0,1]:
      h(\pedro{}x\pedro{})\geq \tilde{h}(x)\right\},
  \end{equation*}
  and the expected payoff following this policy (that is, the optimal expected payoff)
  is
  \begin{equation}\label{eq:16}
    E = (\tilde{h}(\kappa)-h(0 )) e^{\kappa} \pedro{}+h(0).\pedro{}
  \end{equation}

\end{prop}
\begin{proof}
The expected payoff of continuing playing with score $x$ is $G(x)$ as defined in \eqref{eq:Psi} in Lemma \ref{lem:psi-max}. We know that $G(x)>\tilde{h}(x)$ for $x\in [0,\kappa)$, and that $G(x)=\tilde{h}(x)$ for $x\in[\kappa, 1]$.

For any $x\in [0,1]$, a strategy which consists in stopping for $S=x$ is optimal if and only if the
expected payoff when stopping is greater than or equal to the expected payoff when continuing. Thus, stopping is optimal if and only if $h(x)\geq G(x)$.  If $x<\kappa$, then $h(x)<\tilde{h}(x)$, which implies that $h(x)<G(x)$ and the optimal strategy is to continue playing and not stopping. If $x\in (\kappa,1]$, then clearly stopping is optimal because $h(x)\geq \tilde{h}(x)=G(x)$. When $x=\kappa$, stopping is optimal  if and only if $h(\kappa)\geq \tilde{h}(\kappa)$ because  $G(\kappa)=\tilde{h}(\kappa)$.

The expected payoff following the strategy in the statement is, clearly, the expected payoff at the
start of the game: \pedro{}$G (0)=(\tilde{h}(\kappa)-h(0 )) e^{\kappa}+h(0)$\pedro{}, which is $E$.
\end{proof}
Obviously, if $h$ is continuous and $h(x)\neq h(0)$ for some $x\in(0,1)$, then the optimal  threshold $\kappa$   is the
unique root of the equation $h(x)=h(0) x +\int_x^1 h(t)dt$. Notice that, in this case, if at some stage the score is $\kappa$, then both stopping and continuing are optimal strategies.

\pedro{}We provide two simple examples showing how the previous results apply to different payoff functions, one continuous and the other discontinuous.\pedro{}

\begin{exa}
  Assume the payoff in a single-player Showcase Showdown game like above is $h(x)=x$. The
  optimal  threshold is $\kappa=\sqrt{2}-1\simeq 0.41421$, the only root of the equation
  $x=\int_x^1tdt$. The expected payoff is $\kappa e^{\kappa}\simeq 0.62678$.
\end{exa}

\begin{exa}
 Assume now the non-continuous payoff function
\begin{equation*}
h(x):=%
\begin{cases}
x, & \text{if $x<1/2$},\\
7 x, & \text{if }x\geq 1/2.
\end{cases}
\end{equation*}
The optimal threshold  is
\begin{equation*}
  \label{eq:optthre1}
  \kappa = \inf \left\{x \in[0,1]: h(\pedro{}x\pedro{})\geq \int_x^1 h(t)dt\right\} = \frac{1}{2}.
\end{equation*}
\pedro{}The expected payoff is (following the notation of Proposition \ref{OLA}, equation \eqref{eq:16}), $(\tilde{h}(\kappa)-h(0))e^{\kappa}+h(0) = 21e^{1/2}/8 \simeq 4.3278$.
\pedro{}
\end{exa}

\section{Cumulative distribution function of the score.}
We can compute the cumulative distribution function of the random variable
$\xi_{\tau}$ representing the score of a player who follows the threshold
strategy given by the threshold $\tau$.

\begin{lem}
  \label{lem:xi-kappa} Let $\{Z_{n}\}_{n\in\mathbb{N}}$ be a sequence of independent
  uniformly distributed random variables in $[0,1]$. Define
  \begin{equation*}
    \chi_{0}:=0\text{ and }\chi_{n}:=\chi_{n-1}+Z_{n},
  \end{equation*}
  and, for $\tau\in\lbrack0,1]$,
  \begin{equation*}
    \xi_{\tau}:=%
    \begin{cases}
      \chi_{n}, & \text{if $\chi_{n-1} < \tau\leq \chi_{n}\leq1$},\\
      0, & \text{if }\chi_{n-1}< \tau\text{ and }\chi_{n}>1.
    \end{cases}
  \end{equation*}
  Then the cumulative distribution function of $\xi_{\tau}$ is
  \begin{equation}\label{eq:21}
    F_{\tau}(x):= P(\xi_\tau \leq x)=
    \begin{cases}
      0, & \text{if $x<0$},\\
      1+e^{\tau}\,\left(  -1+\tau\right)  , & \text{if $0\leq x\leq \tau$},\\
      1+e^{\tau}\,\left(  -1+x\right)  , & \text{if $\tau<x\leq 1$},\\
      1, & \text{otherwise}.
    \end{cases}
  \end{equation}
\end{lem}

\begin{proof}
  The sequence of random variables above consists in following a threshold strategy with
  greed threshold $\tau$. Thus, we shall reason with a player following such a strategy.

  Assume the player has played $j$ times, for $j\in \mathbb{Z}_{\geq 0}$ and has score
  $S=t$. For $t\leq \tau$, let $f(t)$ be the probability that the final score is
  $S\in [\tau , x]$.
\begin{equation*}
  f(t):=P(\tau\leq\xi_{\tau}\leq x | \chi_{j}=t).
\end{equation*}
In order to compute $f(t)$, we need to consider the two possibilities leading to
$S\in [\tau, x]$ assuming that $\chi_j=t$:
\begin{itemize}
\item[i)] Either $Z_{j+1}\in [\tau-t,x-t]$ (i.e. the process ends after playing once
  more), which happens with probability $x-\tau$.
\item[ii)] More than just another play is required. The probability we are trying to
  compute is the expected value of the probability of ending in $[\tau,x]$ (that is,
  $f(t+s)$) starting with $\chi_j=t+s$ for $s=[0,\tau-t]$):
\begin{equation*}
  \int_{0}^{\tau-t}f(t+s)\,ds.
\end{equation*}
\end{itemize}
Thus,
\begin{equation*}
f(t) = (x-\tau) + \int_{0}^{\tau-t}f(t+s)\,ds,
\end{equation*}
which, as $f$ is necessarily continuous if it satisfies that equation, gives the
differential equation
\begin{equation*}
f^{\prime}(t)=-f(t),
\end{equation*}
with the condition $f(\tau)=x-\tau$, whose solution is $f(t)=e^{\tau-t}\,\left(  x-\tau\right)$. We obtain:
\begin{equation*}
P(\tau\leq\xi_{\tau}\leq x)=f(0)=e^{\tau}(x-\tau).
\end{equation*}
Therefore, the  cumulative distribution function of $\xi_{\tau}$ is given, for $x\in[0,\tau]$, by:
\begin{equation}\label{eq:27}
F_{\tau}(x)=P(0\leq \xi_{\tau}\leq \tau)=1-P(\tau\leq\xi_{\tau}\leq1)=1-e^{\tau}(1-\tau),
\end{equation}
while, for $x\in[\tau,1]$, we have:
\begin{align}\label{eq:28}
  F_{\tau}(x)&
             =P(0\leq\xi_{\tau}\leq x)=
             P(0\leq\xi_{\tau}\leq \tau)+
             P(\tau\leq\xi_{\tau}\leq x)
             =1+e^{\tau}\,\left(  -1+x\right),
\end{align}
and the rest of the statement follows.
\end{proof}

\begin{rem}
 \label{rem:P} Notice that $F_{\tau}(x)$ is the distribution of the product of two independent  random variables, one
of which is Bernoulli with parameter $e^{\tau }(1-\tau )$, and the other uniform on $[\tau ,1]$:
\begin{equation*}
\xi_{\tau }\sim    \verb"Be"(e^{\tau }(1-\tau )) \mathbf{U}
[\tau ,1].
\end{equation*}
The  \verb"success" event in the Bernoulli random variable $\verb"Be"(e^{\tau}\,\left(  1-\tau\right)  )$ means getting a payoff greater than $0$ using $\tau$ as greed threshold, which is the same as obtaining $S\in [\kappa,1]$ with uniform distribution for $S$ in that interval.

 Observe that if $\tau \in [0,1]$, then by \eqref{eq:27} and \eqref{eq:28}, we obtain $P(0\leq\xi_{\tau}\leq\tau)=P(\xi_{\tau}=0)$. In what follows, we shall use the
 notation
 \begin{equation}\label{eq:30}
   \frakp{}(x):=P(\xi_{x}=0)=1+e^{x}\,\left( -1+x\right).
 \end{equation}
 to represent the probability of a player with greed threshold $x$ to get score $S=0$.
\end{rem}
The following easy corollary will be useful.
\begin{cor}
  \label{cor:exp} Let $h$ be a Riemann integrable function on
$[0,1]$. Then,

\begin{itemize}
\item[i)] The expected value $\mathbb{E}[h(\xi_x)]$ satisfies:
  \begin{equation*}\displaystyle \mathbb{E}\left[  h(\xi_{x})\right] =
    \frakp{}(x)h(0)+e^{x}\int_{x}^{1}\pedro{h}\pedro{}(t) dt.
  \end{equation*}
\item[ii)] The conditional expected value
  $\mathbb{E}\left[ h(\xi_{x})|\xi_{x}>0\right]$ satisfies:
  \begin{equation}\label{eq:32}
    \displaystyle\mathbb{E}\left[ h(\xi_{x})|\xi_{x}>0\right]
    =\int_{x}^{1}\frac{h(t)}{1-x}dt, x\in [0,1).
  \end{equation}
\end{itemize}
\end{cor}
\begin{proof}
  Since $\xi_{x}\sim \verb"Be"( e^{x}\,\left( 1 - x\right) )\cdot\mathbf{U}[x,1]$, we have
that
  \begin{equation*}
  \mathbb{E}[h(\xi_x)] = \frakp{}(x) h(0)+ (1- \frakp{}(x)) \int_{x}^{1}
\frac{h(t)}{1-x}dt= \frakp{}(x) h(0)+ e^x \int_{x}^{1} h(t) dt,
  \end{equation*}
while
\begin{equation*}
\mathbb{E}[h(\xi_x)| \xi_x>0] = \mathbb{E}[h(\mathbf{U}[x,1])] = \int_{x}^{1} \frac{h(t)}{1-x} dt.
\end{equation*}
\end{proof}

\section{ Game I. Sequential game}
 In this game, there are $n$ players which play just once each,
sequentially. Each one knows the score obtained by the previous ones. The
winner receives a unitary payoff, and it is irrelevant whether this comes from an external payer or from the other players. Our aim is to find the optimal policy for each player, that is, the threshold strategy which maximizes his expected payoff ---or, what is the same, his probability of winning.

The Showcase Showdown game without draw is studied in \cite{Seregrina}:\emph{ if all the scores exceed 1, then the winner is the player with the lowest one}. In the sequential game, if all previous players had score 0, then the last player (assuming a rational behavior) wins simply stopping after the first play. Thus, the present section contains parts of section 3.2 of that reference, in which the optimal policy for each player is provided, as well as the probability of the first player winning (but just the first one). In this work, using some delicate arguments, we have been able to define a recursive procedure for computing each players' probability of winning, assuming all of them act optimally. We also include, for the case of three players, a study of the possible coalitions that can be made that increase the joint probability of winning for the coalesced.

\pedro{}
\begin{lem}
\label{lem:theta}
  For all $n>0$, the equation $ \frakp{}(x)^{n-1}=  \int_x^1 \frakp{}(t)^{n-1}dt$, i.e,
\begin{equation}\label{eq:35}
{\left(  1-e^{x}+e^{x}\,x\right)  }^{n-1}=\int_{x}^{1}{\left(  1-e^{t}%
+e^{t}\,t\right)  }^{n-1}\,dt, 
\end{equation}
has a single solution in  $[0,1]$, which will be denoted $\theta_{n}$
\end{lem}
\pedro{}
  \begin{proof}
Put $f_n(x)=\left(  1-e^{x}+e^{x}\,x\right)^{n-1}$ and $g_n(x)=\int_{x}^{1}{\left(  1-e^{t}%
+e^{t}\,t\right)  }^{n-1}\,dt$. If $n=1$, then $f_1(x)=1$ and $g_1(x)=1-x$ so $x=0$ is the only solution. Now, for $n>1$, we have that $f_n(0)=0$, $f_n(1)=1$, $g_n(0)>0$ and $g_n(1)=0$. Since $f_n$ is increasing and $g_n$ is decreasing in $[0,1]$ for every $n>0$, the result follows.
\end{proof}
\pedro{}
\begin{lem}
  \label{lem:theta-increasing} The sequence $\{\theta_{n}\}_{n>0}$
is strictly increasing.
\end{lem}

 \begin{proof}
With the previous notation, it is enough to show that $f_{n+1}(\theta_n)<g_{n+1}(\theta_n)$. Then the result follows just like in Lemma \ref{lem:theta}. To do so, first note that
\begin{equation*}
  (1-e^{\theta_n}+\theta_n e^{\theta_n})(1-e^t+te^t)^{n-1}<(1-e^t+te^t)^n,
\end{equation*}
for every $t\in(\theta_n,1]$. Then,
\begin{equation*}
  f_{n+1}(\theta_n)=(1-e^{\theta_n}+\theta_ne^{\theta_n})f_n(\theta_n)=(1-e^{\theta_n}+\theta_ne^{\theta_n})g_n(\theta_n)<g_{n+1}(\theta_n),
\end{equation*}
and the claim follows.
\end{proof}

\begin{prop}
  \label{estrategia} In Game $I$, let us assume that there are $r$
players still to play. If $M_r$ is the maximum score of the players who have
already played, then the optimum threshold for the next player is
$\max\{\theta_{r},M_r\}$.
\end{prop}
\begin{proof}
We argue by induction on $r$. The case $r=1$ is trivial because the optimum threshold of the last player is obviously $M_1=\max\{\theta_1,M_1\}$ (notice that, as $\theta_n$ satisfies \eqref{eq:35}, we have $\theta_1=0$). Thus, let $r>1$ and assume that the result holds for $r-1$.

We will call $h_r(x)$ the probability of winning stopping with score $S=x$. By induction hypothesis, $h_r(x)$ is the probability that the following players finish with a score of 0 when trying to exceed their respective optimal thresholds. And, obviously, if $x\leq M_r$ the probability of winning is 0, so

\begin{equation*}
h_r(x)=\left\{
\begin{array}
[c]{ccc}%
\prod_{s=1}^{r-1} \frakp{}(\max(\theta_{s},x)) & \text{if} & x>M_r, \\[0.5em]%
0 & \text{if} & M_r\geq x.
\end{array}
\right.
\end{equation*}

Now, to be in the conditions of Proposition \ref{OLA}, we must see that

\begin{equation*}
  \max(M_r,\theta_r)=\inf \left\{x \in[0,1]: h_r(x)\geq  \int_x^1 h_r(t)dt\right\}.
\end{equation*}

$\bullet$ If $M_r<\theta_r$, then, $h_r(x)=\frakp{}(x)^{r-1}$  for all  $x>\theta_r$. Thus, taking into account   that   $ \frakp{}(\theta_r)^{r-1}=  \int_{\theta_r}^1 \frakp{}(t)^{r-1}dt$, it follows that

\begin{equation*}
  \inf \left\{x \in[0,1]: h_r(x)\geq  \int_x^1 h_r(t)dt\right\}=\inf \left\{x \in[0,1]:  \frakp{}(x)^{r-1}\geq  \int_x^1 \frakp{}(x)^{r-1}dt\right\}=\theta_r,
\end{equation*}
and ultimately

\begin{equation*}
  \inf \left\{x \in[0,1]: h_r(x)\geq  \int_x^1 h_r(t)dt\right\}=\max(M_r,\pedro{}\theta_r\pedro{}).
\end{equation*}

$\bullet$ If $M_r\geq\theta_r$ then $h_r(x)\geq  \int_x^1 h_r(t)dt $ for all $x \geq  M_r$ and $0=h_r(x)<  \int_x^1 h_r(t)dt$ for all $x<M_r$ so that

\begin{equation*}
  \max(M_r,\theta_r)=M_r = \inf \left\{x \in[0,1]: h_r(x)\geq  \int_x^1 h_r(t)dt\right\}.
\end{equation*}
\end{proof}

\begin{rem}
 \label{est1} In Proposition \ref{estrategia}, if $r=n$ (i.e. it
is the first player's turn) we will obviously consider $\pedro{}M_0\pedro{}=0$. Therefore, the
optimum threshold for the first player is $\theta_{n}$.
\end{rem}

  The following consequence is obvious:

\begin{cor}
  For every $1\leq r\leq n$, if there are $r$ players still to play
and all the previous ones got a score less than $\theta_{r}$, then the optimum threshold for the next player (the one whose turn it is) is $\theta_{r}%
$.
\end{cor}

  In what follows, we will assume that all the players follow their
\emph{optimal strategy} described in Proposition \ref{estrategia}.
Given $1\leq r <n$, $1\leq m \leq r$, and $x\geq \theta_r$, we shall denote by $F_r^m(x)$ the winning probability of the $(n-r+m)$-th player when the maximum score of the first $n-r$ players is $x$.
\begin{prop}
  \label{pro:fnm}
  The following equalities hold:
\begin{equation*}
F_{r}^{m}(x)=\left\{
\begin{array}
[c]{ccc}%
e^{x}\,\int_{x}^{1}{\left(  1-e^{t}+e^{t}\,t\right)  }^{r-1}\,dt & \text{if} &
m=1, \\[0.5em]%
{\left(  1-e^{x}+e^{x}\,x\right)  }F_{r-1}^{m-1}(x)\,+e^{x}\int_{x}^{1}%
F_{r-1}^{m-1}(t)dt & \text{if} & m\geq 2.
\end{array}
\right.
\end{equation*}
\end{prop}
  \begin{proof}
For $m=1$, in order for the $(n-r+1)$-th player to win, all the remaining $r-1$ players must obtain a score of $0$. We know that, if $\xi_x\neq 0$, then $\xi_x>x\geq \theta_r$. Consequently, due to Lemma \ref{lem:theta-increasing} all the subsequent players will use $\xi_x$ as their greed threshold because of Proposition \ref{estrategia}. Thus, using Corollary \ref{cor:exp} and \eqref{eq:30}, we get that
\begin{equation*}
  F_{r}^{1}(x)=\mathbb{E}[\frakp{}(\xi_x)^{r-1}]=
  e^{x}\,\int_{x}^{1}{\left(  1-e^{t}+e^{t}\,t\right)  }%
^{r-1}\,dt.\end{equation*}
Now, let us assume that $m>1$ and denote by $W_m$ the event ``the $(n-r+m)$-th player wins''. Since the $(n-r+1)$-th player used $x$ as optimum threshold, his score is $\xi_x$. Then, the Law of total probability gives:
\begin{equation*}
F_r^m(x)=P(W_{m})=P(W_m|\xi_x=0)P(\xi_x=0) + P(W_m|\xi_x>0)P(\xi_x>0).
\end{equation*}
Furthermore, on one hand, it is straightforward that
\begin{equation*}
P(W_m|\xi_x=0)=F_{r-1}^{m-1}(x),
\end{equation*}
while on the other,
\begin{equation*}
P(W_m|\xi_x>0)=\mathbb{E}[(F_{r-1}^{m-1}(\xi_{x})|\xi_x>0],
\end{equation*}
because, in this case, the maximum score obtained by the first $n-r+1$ players is $\xi_x\geq\theta_{r-1}$.
Combining both equalities and using Corollary \ref{cor:exp} we obtain the desired formula:
\begin{multline*}
F_{r}^{m}(x)=\frakp{}(x) F_{r-1}^{m-1}(x)+(1-\frakp{}(x)) \int_{x}%
^{1}\frac{F_{r-1}^{m-1}(t)}{1-x}dt=\\
{\left(  1-e^{x}+e^{x}\,x\right)  }F_{r-1}^{m-1}(x)\,+e^{x}\int_{x}^{1}%
F_{r-1}^{m-1}(t)dt.
\end{multline*}
\end{proof}

  With this proposition, we are in the condition to prove the main
result of this section. Let us denote by $P_{n}^{m}$ the winning probability
of the $m$-player in Game I if there are $n$ players. Then, we have the
following.

\begin{prop}
 \label{prop:pnm} The winning probability of player $m$ in Game I,
$1\leq m\leq n,$ using the optimal policy described in Proposition
\ref{estrategia} is:
\begin{equation*}
P_{n}^{m}=\left\{
\begin{array}
[c]{ccc}%
e^{\theta_{n}}\left(  1-e^{\theta_{n}}+e^{\theta_{n}}\ \theta_{n}\right)
^{n-1} & \text{if} & m=1, \\
\left(  1-e^{\theta_{n}}+e^{\theta_{n}}\ \theta_{n}\right)  P_{n-1}%
^{m-1}+e^{\theta_{n}}\int_{\theta_{n}}^{1}F_{n-1}^{m-1}(t)\,dt & \text{if} &
m\geq 2.
\end{array}
\right.
\end{equation*}

\end{prop}

\begin{proof}
  \pedro{}
  The proof is essentially the same as for Proposition \ref{pro:fnm}, using $\xi_{\theta_n}$ as greed thresholds for each player, and expected values instead of distribution functions.
\pedro{}
\end{proof}

\pedro{}
\begin{rem}
  Assume none of the first $n-r$ players got a score higher than $\theta_r$. Then the winning probability of player $n-r+1$ is the same as the winning probability of the first player in the $r$-player game.
\end{rem}
\pedro{}

Table \ref{tab:threshold-10} shows the optimal greed threshold
$\theta_{n}$ of the first player in an $n$-player game (or when there are
still $n$ players and none of the previous ones got a positive score), and the
winning probabilities $P_{n}^{m}$ of the $m$-th player in Game I, for
$n=1,\ldots, 10$.
\begin{table}[ptb]
  \centering
\caption{Approximate values of $\theta_{n}$ and winning probabilities ($n\leq10$) for each player.}%
\label{tab:threshold-10}%
\begin{tabular}
[c]{cccccccccc}%
$n$ & $2$ & $3$ & $4$ & $5$ & $6$ & $7$ & $8$ & $9$ & $10$\\\hline
\multicolumn{1}{c}{$\theta_{n}$} & \multicolumn{1}{c}{$0.5706$} &
\multicolumn{1}{c}{$0.6879$} & \multicolumn{1}{c}{$0.7487$} &
\multicolumn{1}{c}{$0.7871$} & \multicolumn{1}{c}{$0.8141$} &
\multicolumn{1}{c}{$0.8342$} & \multicolumn{1}{c}{$0.8499$} &
\multicolumn{1}{c}{$0.8626$} & \multicolumn{1}{c}{$0.8730$}\\\hline
\multicolumn{1}{c}{$P_{n}^{1}$} & \multicolumn{1}{c}{$0.4250$} &
\multicolumn{1}{c}{$0.2859$} & \multicolumn{1}{c}{$0.2176$} &
\multicolumn{1}{c}{$0.1764$} & \multicolumn{1}{c}{$0.1486$} &
\multicolumn{1}{c}{$0.1285$} & \multicolumn{1}{c}{$0.1133$} &
\multicolumn{1}{c}{$0.1013$} & \multicolumn{1}{c}{$0.0917$}\\\hline
\multicolumn{1}{c}{$P_{n}^{2}$} & \multicolumn{1}{c}{$0.5750$} &
\multicolumn{1}{c}{$0.3248$} & \multicolumn{1}{c}{$0.2357$} &
\multicolumn{1}{c}{$0.1866$} & \multicolumn{1}{c}{$0.1551$} &
\multicolumn{1}{c}{$0.1329$} & \multicolumn{1}{c}{$0.1165$} &
\multicolumn{1}{c}{$0.1037$} & \multicolumn{1}{c}{$0.0936$}\\\hline
\multicolumn{1}{c}{$P_{n}^{3}$} & \multicolumn{1}{c}{} &
\multicolumn{1}{c}{$0.3893$} & \multicolumn{1}{c}{$0.2570$} &
\multicolumn{1}{c}{$0.1978$} & \multicolumn{1}{c}{$0.1619$} &
\multicolumn{1}{c}{$0.1375$} & \multicolumn{1}{c}{$0.1197$} &
\multicolumn{1}{c}{$0.1061$} & \multicolumn{1}{c}{$0.0954$}\\\hline
\multicolumn{1}{c}{$P_{n}^{4}$} & \multicolumn{1}{c}{} &
\multicolumn{1}{c}{} & \multicolumn{1}{c}{$0.2897$} &
\multicolumn{1}{c}{$0.2104$} & \multicolumn{1}{c}{$0.1691$} &
\multicolumn{1}{c}{$0.1422$} & \multicolumn{1}{c}{$0.1230$} &
\multicolumn{1}{c}{$0.1085$} & \multicolumn{1}{c}{$0.0972$}\\\hline
\multicolumn{1}{c}{$P_{n}^{5}$} & \multicolumn{1}{c}{} &
\multicolumn{1}{c}{} & \multicolumn{1}{c}{} & \multicolumn{1}{c}{$0.2289$}
& \multicolumn{1}{c}{$0.1770$} & \multicolumn{1}{c}{$0.1470$} &
\multicolumn{1}{c}{$0.1263$} & \multicolumn{1}{c}{$0.1109$} &
\multicolumn{1}{c}{$0.0991$}\\\hline
\multicolumn{1}{c}{$P_{n}^{6}$} & \multicolumn{1}{c}{} &
\multicolumn{1}{c}{} & \multicolumn{1}{c}{} & \multicolumn{1}{c}{} &
\multicolumn{1}{c}{$0.1883$} & \multicolumn{1}{c}{$0.1523$} &
\multicolumn{1}{c}{$0.1297$} & \multicolumn{1}{c}{$0.1133$} &
\multicolumn{1}{c}{$0.1008$}\\\hline
\multicolumn{1}{c}{$P_{n}^{7}$} & \multicolumn{1}{c}{} &
\multicolumn{1}{c}{} & \multicolumn{1}{c}{} & \multicolumn{1}{c}{} &
\multicolumn{1}{c}{} & \multicolumn{1}{c}{$0.1596$} &
\multicolumn{1}{c}{$0.1333$} & \multicolumn{1}{c}{$0.1158$} &
\multicolumn{1}{c}{$0.1026$}\\\hline
\multicolumn{1}{c}{$P_{n}^{8}$} & \multicolumn{1}{c}{} &
\multicolumn{1}{c}{} & \multicolumn{1}{c}{} & \multicolumn{1}{c}{} &
\multicolumn{1}{c}{} & \multicolumn{1}{c}{} & \multicolumn{1}{c}{$0.1382$}
& \multicolumn{1}{c}{$0.1184$} & \multicolumn{1}{c}{$0.1044$}\\\hline
\multicolumn{1}{c}{$P_{n}^{9}$} & \multicolumn{1}{c}{} &
\multicolumn{1}{c}{} & \multicolumn{1}{c}{} & \multicolumn{1}{c}{} &
\multicolumn{1}{c}{} & \multicolumn{1}{c}{} & \multicolumn{1}{c}{} &
\multicolumn{1}{c}{$0.1218$} & \multicolumn{1}{c}{$0.1063$}\\\hline
\multicolumn{1}{c}{$P_{n}^{10}$} & \multicolumn{1}{c}{} &
\multicolumn{1}{c}{} & \multicolumn{1}{c}{} & \multicolumn{1}{c}{} &
\multicolumn{1}{c}{} & \multicolumn{1}{c}{} & \multicolumn{1}{c}{} &
\multicolumn{1}{c}{} & \multicolumn{1}{c}{$0.1088$}\\\hline
\end{tabular}
 \end{table}

\begin{rem}
  The \pedro{}threshold strategies described in Proposition \ref{estrategia}\pedro{} constitute a Nash
  equilibrium: no unilateral deviation will increase the winning probability of the
  defector. However, \pedro{}though optimal\pedro{} for each single player, \pedro{}each
  value $P^r_n$ is the winning probability for player $r$\pedro{} only if every player
  plays optimally according to his interests without mistakes or collaboration (collusion)
  within some group.  The deviation of a group of players from their individual optimum
  policy can increase the winning probability of some players and decrease that of
  others. To illustrate this, we now carry out a detailed study of the 3-player game, in
  which collusion between two players can modify the winning probabilities of the
  other one.
\end{rem}
\pedro{}Notice, however, that the winning probability of each player using his optimal threshold $\theta_i$ (shown in Table 1) does not represent the value of the game for each, as their winning probabilities depend on the behavior of the other players. There exist coalitions, already in the $3$-player game which decrease the winning probability of the remaining player, thus increasing the probability of one of the allies win. This is what we study in the next two subsections.\pedro{}

\subsection{Coalition between first and second player}
The first and second players can decrease the third one's winning probability using some strategy. Let $\ttx{}(x) $
be the second player's optimal threshold conditional to the first one having obtained score $x$, assuming their common aim is to decrease the third  player's winning chance. Let
$\mathbf{p}_{3} (x)$ be the third player's losing probability assuming the first one has obtained score $x$ and the second one uses threshold
$\ttx{}(x)$, which is:
\begin{equation}\label{eq:50}
\mathbf{p}_{3} (x)=
P(\xi_{\ttx{}(x)}=0)\cdot P(\xi_{x}=0)+
P(\xi_{\ttx{}(x)}>0)
\cdot \mathbb{E}
(\frakp{}
(\xi_{\ttx{}(x)})|\xi_{\ttx{}(x)}>0),
\end{equation}
and by Corollary \ref{cor:exp}:
\begin{equation*}
\mathbf{p}_{3} (x)= \frakp{}(\ttx{}(x)) \frakp{}(x) + e^{ \ttx{}(x) } \int_{\ttx{}(x)}^{1} \frakp{}(t) dt.
\end{equation*}

Assume the first player has score $x$ and plays once more. The probability of the third player not winning under this assumption is given by:
\begin{equation*}
x\cdot\mathbf{p}_{3} (0) + \int_{x}^{1} \mathbf{p}_{3} (t) dt.
\end{equation*}
Thus, by Proposition \ref{OLA}, the first player's optimal threshold
$\Theta_{1}$, satisfies:
\begin{equation*}
 \mathbf{p}_{3} (\Theta_{1})=\Theta_{1} \mathbf{p}_{3} (0) + \int_{\Theta_{1}}^{1} \mathbf{p}_{3} (t) dt,
\end{equation*}
and the probability of the third player not winning is by the same Proposition \ref{OLA}:
\begin{equation*}
(\mathbf{p}_{3}(\Theta_{1})-\mathbf{p}_{3}(0 )) e^{\Theta_{1}} + \mathbf{p}_3(0).
\end{equation*}
Computing the value of $\Theta_{1}$ is very laborious, as one needs to calculate  $\ttx{}(x)$ and $\mathbf{p}_{3} (x)$ beforehand.

\paragraph{Calculation of $\ttx{}(x)$}
If the first player has stopped with score $x$ and the second with $y>x$, the probability that the third one does not win is $\frakp{}(y)$. Applying Proposition \ref{OLA}, $\ttx{}(x)$ is the solution of the equation:
\begin{equation*}
\frakp{}(y)= y \frakp{}(x)+ \int_{y}^{1} \frakp{}(t)dt,
\end{equation*}
which together with \eqref{eq:30},
gives $\ttx{}(x)$ implicitly:
\begin{equation*}
-e^{\ttx{}(x)}(2\ttx{}(x)-3)+\ttx{}(x)e^{x}(x-1)=e.
\end{equation*}
Foregoing the uninteresting details, we have obtained
:
\begin{equation*}
\Theta_{1}=0.63386...
\end{equation*}
so that the third player's winning probability is:
\begin{equation*}
1- (\mathbf{p}_{3}(\Theta_{1})-\mathbf{p}_{3}(0 )) e^{\Theta_{1}} - \mathbf{p}_3(0)
= 0.3867...
\end{equation*}
slightly less than $0.3893$ ($P_3^3$ in Table 1), which is the one under the Nash equilibrium.

\subsection{Coalition between first and third player}
There is also a possible coalition between the first and third players harmful to the second one. In this case, however, the third player's only option is to try and improve on the second one's score, which would become his playing threshold in this strategy. The second player's optimal strategy is the one described in Proposition \ref{estrategia}, and consists in using the threshold
$\max(\theta_{2},X_{1})$, where $X_{1}$ is the first player's score.

The second player's winning probability assuming the first one's score is $x\geq\theta_{2}$, is given by:
\begin{equation*}
\widehat{h}(x)= (1-\frakp{}(x)) \int_{x} ^{1} \frac{ \frakp{}(t) }{1-x}dt=
e^{x} \int_{x} ^{1}\frakp{}(t) dt=-e^{x} \left( e^{x} (x-2)+x+e-1\right)
.
\end{equation*}
The strategy under discussion aims to lower this value, which requires finding a new greed threshold for the first player.

If the first player has score $x$ and plays once more, then the probability of the second player not winning is  $x (1-\vartheta)+ \int_{x}^{1} (1-\widehat{h}(t))
dt$, where $\vartheta=0.4250...$. In this formula, $\vartheta$ is the winning probability of the second player if he played just against the third one (what in Table 1 is $P_2^1$). Setting
$h(t):=1-\widehat{h}(t)$, the first player's required greed threshold will be given by $\varrho$ satisfying, by Proposition \ref{OLA}:
\begin{equation*}
h(\varrho)= \varrho(1-\vartheta)+ \int_{\varrho}^{1} h(t) dt,
\end{equation*}
whose value is, approximately,
\begin{equation*}
\varrho = 0.75017...
\end{equation*}
and the second player's winning probability assuming the first one uses this threshold is:
\begin{equation*}
\frakp{}(\varrho) \vartheta+(1-\frakp{}(\varrho) )
\int_{\varrho}^{1} \frac{ \widehat{h}(t) }{1-\varrho}dt=0.32262 ...
\end{equation*}
slightly less than $P_3^2=0.3248$ (Table 1), which is the second player's winning probability in the Nash equilibrium.

\section{Game II --- No-information game}

 In this version, the $n$ players act simultaneously with no
information on the results of the others. Unlike the studies \cite{Mazalov, Seregrina}, we
consider the no-winner possibility, that is, there is a global tie when all
the players get a score of $0$. We study in this section the varieties
introduced in Section 2: For Games II.1 and II.2, if all the players get a
score of $0$, there is no payoff, while Game II.3 is an asymmetric version in
which a single player has advantage: he wins the payoff in case of global tie
at $0$. Throughout the section the number of players is denoted, as above, by
$n$.

\subsection{  Game II.1: Non-constant sum}

  In this version, if there is a winner, he receives a payoff of
$1$ from an external agent. Hence, each player's expected payout is his
probability of winning.

\pedro{}
\begin{lem}
 \label{lem:alpha} For all $n>0$, the equation
\begin{equation}\label{eq:64}
\left(  1+e^{x}(x-1)\right)  ^{n-1}=\dfrac{1-\left(  1+e^{x}(x-1)\right)
^{n}}{ne^{x}}, 
\end{equation}
has a single solution $\alpha_{n}\in (0,1)$.
\end{lem}
\begin{proof}
Put $f_n(x)=\left(  1+e^x(x-1)\right)  ^{n-1}$ and $g_n(x)=\dfrac{1-\left(1+e^x(x-1)\right)^n}{ne^x}$. If $n=1$, then $f_1(x)=1$ and $g_1(x)=1-x$ so $x=0$ is the only solution. Now, for $n>1$, we have that $f_n(0)=0$, $f_n(1)=1$, $g_n(0)=1/n>0$ and $g_n(1)=0$. Since $f_n$ is increasing and $g_n$ is decreasing in $[0,1]$ for every $n>1$, the result follows.
\end{proof}
\pedro{}
\begin{lem}
 \label{lem:alpha-increasing} The sequence $\{\alpha_{n}\}_{n>0}$ from Lemma \ref{lem:alpha}
is strictly increasing.
\end{lem}

  \begin{proof}
    With the previous notation, it is enough to show that $f_{n+1}(\alpha_n)<g_{n+1}(\alpha_n)$. First note that $f_{n+1}(\alpha_n)=\left(1+e^{\alpha_n}(\alpha_n-1)\right)f_n(\alpha_n)$ and by definition of $\alpha_n$ (Lemma \ref{lem:alpha}),
    \begin{equation*}
      f_{n+1}(\alpha_n)=\left(1+e^{\alpha_n}(\alpha_n-1)\right)g_n(\alpha_n).
    \end{equation*}
    Then, after some straightforward computations, we see that $f_{n+1}(\alpha_n)<g_{n+1}(\alpha_n)$ if and only if
\begin{equation*}(n+1)\left(1+e^{\alpha_n}(\alpha_n-1)\right)-\left(1+e^{\alpha_n}(\alpha_n-1)\right)^{n+1}<n.\end{equation*}
Now, for every \pedro{}$t\in(0,1)$\pedro{}, we have that $(n+1)t-t^{n+1}<n$ \pedro{}because the left-hand side is an increasing function in $t$, whose value for $t=1$ is $n$\pedro{}, and we obtain the result.
\end{proof}

\begin{prop}
  \label{pro:equilibrium-game-2-1} Game II.1 admits a Nash
equilibrium with equal thresholds $\alpha_{n}$ for all players. Moreover, the
winning probability of each player in that Nash equilibrium is
\begin{equation}\label{eq:67}
P_{n}:=\frac{1-\left(  1-e^{\alpha_{n}}+e^{\alpha_{n}}\,{\alpha_{n}}\right)
^{n}}{n}.%
\end{equation}
For this game, we shall call \emph{Nash threshold} the value $\alpha_{n}$.
\end{prop}

  \begin{proof}
  Assuming that all rivals use thresholds $\alpha_n$, we will call $\xi_{\alpha_n}^{(i)}$ the random variable that represents the score obtained by the $i$-th player using the threshold $\alpha_n $ and $h(x)$ the probability of winning for a player when stopping with score $x$. We will assume, without loss of generality, that the reference player is the first and that the rest use a greed threshold $\alpha_n$. It is about proving that, under these conditions, the optimal threshold for the first player is precisely $\alpha_n$.

  $\bullet$ If $x \geq\alpha_n$, then, by Lemma \ref{lem:xi-kappa}:
\begin{equation*}
h(x)= P(x>\xi_{\alpha_n}^{(i)}: i=2,\dots,n)=\prod_{i=2}^{n} P(x>\xi_{\alpha_n}^{(i)})=( 1 + e^{\alpha_n}\,\left(  -1 + x\right)  )^{n-1}.
\end{equation*}

$\bullet$ If $x<\alpha_n$, then
\begin{equation*}
h(x)= P(x>\xi_{\alpha_n}^{(i)}: i=2,\dots,n)=\prod_{i=2}^{n} P(\xi_{\alpha_n}^{(i)}=0)=\frakp{}(\alpha_n)^{n-1}.
\end{equation*}
In any case, {$h(x)$ is continuous and non-decreasing in $[0,1]$ so that by Proposition \ref{OLA} it is enough to verify that \begin{equation*}h(\alpha_n)=\int_{\alpha_n}^ 1 h (t) dt.\end{equation*}}

We have, by definition of $\alpha_n$, that
\begin{equation*}
  ( 1 + e^{\alpha_n}\,\left(  -1 +\alpha_n \right)  )^{n-1}=\frac{
    1-\left(e^{\alpha_n} (\alpha_n-1)+1\right)^n }{n e^{\alpha_n}},
\end{equation*}
And finally, one just needs to keep in mind that
\begin{equation*}
  ( 1 + e^{\alpha_n}\,\left(  -1 +\alpha_n \right)  )^{n-1}=h(\alpha_n),
\end{equation*}
\begin{equation*}
  \frac{  1-\left(e^{\alpha_n} (\alpha_n-1)+1\right)^n }{n e^{\alpha_n}}=\int_{\alpha_n}^ 1 h(t) dt.
\end{equation*}
In this equilibrium, since all players have the same probability of winning, the winning probability of each player is:
\begin{equation*}
P_{n}=\frac{1-P(\xi_{\alpha_{n}}=0)^{n}}{n}=\frac{1-\left(  1-e^{\alpha_{n}}+e^{\alpha_{n}}\,{\alpha_{n}}\right)^{n}}{n}.
\end{equation*}

\end{proof}

 Table \ref{tab:game-2-1} contains the Nash thresholds and the
winning probabilities for each player in this version of the $n$-player game
(no-information and external payer). \pedro{}A simple example follows for $n=2$.\pedro{}
\begin{table}[h]
  \centering
\caption{Approximate Nash threshold and winning probability for Game II.1}%
\label{tab:game-2-1}%
\begin{tabular}
[c]{cccccccccc}%
$n$ & $2$ & $3$ & $4$ & $5$ & $6$ & $7$ & $8$ & $9$ & $10$\\\hline
{${\alpha_{n}}$} & {$0.5887$} & {$0.6989$} & {$0.7562$} & {$0.7927$} &
{$0.8184$} & {$0.8377$} & {$0.8528$} & {$0.8650$} & {$0.8751$}\\\hline
{$P_{n}$} & {$0.4665$} & {$0.3129$} & {$0.2366$} & {$0.1907$} & {$0.1598$} &
{$0.1375$} & {$0.1208$} & {$0.1077$} & {$0.0972$}\\\hline
\end{tabular}
\end{table}

\begin{exa}
  For $n=2$, the Nash threshold is $0.5887+$, and $P_{2}=0.4665+$.
Notice, however, that the use of this threshold does not guarantee a winning
probability of $0.4665$: the other player abandoning the equilibrium strategy
can be harmful for \emph{both of them}, as we will now show.

When the players use respective thresholds $x$ and $y$, then the
winning probability of the first player, $P_1(x,y)=P(\xi_x>\xi_y),$ is:
$$
P_1(x,y)=P(\xi_y=0)P(\xi_x>0) + P(\xi_y>0) P(\xi_x>0) P(\xi_x>\xi_y |\xi_x>0
\wedge \xi_y>0 ),
$$
that is, 
$$
P_1(x,y)=\frakp{}(y)(1-\frakp{}(x))+(1-\frakp{}(y))(1-\frakp{}(x)) P(\xi_x>\xi_y
|\xi_x>0 \wedge \xi_y>0 )
$$
which, recalling \eqref{eq:30}, gives:

If $x\leq y$, then
$$
P(\xi_x>\xi_y |\xi_x>0 \wedge \xi_y>0 )
=P(\textbf{U}[x,1]>\textbf{U}[y,1])=\frac{1-y}{2 (1-x)}
$$
so that:
$$
P_1(x,y)=\frac{1}{2}e^{x}\left(  e^{y}(y-1)(-2x+y+1)-2x+2\right)
$$

Otherwise, if $x> y$, then
$$
P(\xi_x>\xi_y |\xi_x>0 \wedge \xi_y>0 )
=P(\textbf{U}[x,1]>\textbf{U}[y,1])=\frac{x-y}{1-y}+\frac{1-x}{2 (1-y)}
$$
so that:
$$
P_1(x,y)=-\frac{1}{2}e^{x}(x-1)\left(  (x-1)e^{y}+2\right)
$$
In short:
\begin{equation}
P_{1}(x,y)=%
\begin{cases}
\displaystyle\frac{1}{2}e^{x}\left(  e^{y}(y-1)(-2x+y+1)-2x+2\right)  , &
\mathrm{if}\;x\leq y\\[0.5em]%
-\displaystyle\frac{1}{2}e^{x}(x-1)\left(  (x-1)e^{y}+2\right)  , &
\mathrm{otherwise}.
\end{cases}
\end{equation}
If the first player chooses threshold $\alpha_{2}=0.5887+$, then it is easy to verify that, for $y\in(\alpha_{2},0.745+)$,
\begin{equation}
P_{1}(\alpha_{2},y)<P_{1}(\alpha_{2},\alpha_{2}).
\end{equation}
Which means that the first player's expected payoff can decrease
depending on the greed threshold chosen by the second one (see Figure 1).
\begin{figure}[h]
  \centering
  \includegraphics[width=6in]{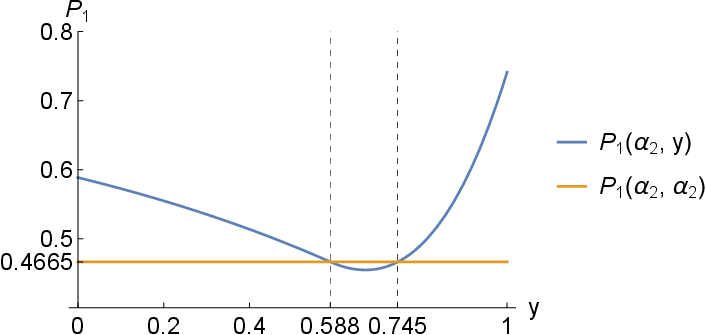}
  \caption{Winning probability of Player 1 (blue) depending on the
    threshold of Player $2$ ($y$). In orange: winning probability of each player in
the Nash equilibrium.}
\end{figure}
\end{exa}

\begin{rem}
 This equilibrium is not a strong equilibrium: the trivial
all-player coalition using the same greed threshold $0$ provides an expected
 payoff of $1/n$, which is greater than the one of Proposition \ref{pro:equilibrium-game-2-1}. However, this is an unstable
coalition as each player would be incentivized to abandon it unilaterally in
order to obtain a payoff greater than $1/n$. We believe (but have no proof)
that the Nash equilibrium of Proposition \ref{pro:equilibrium-game-2-1} is
coalition-proof, and only binding collaboration agreements can ensure a
greater payoff. This Nash equilibrium is not ``foolproof'' either, in the
sense that if some players abandon that strategy (knowingly or by mistake),
they may decrease not only their expected payoff but also that of other
players, as Example 3 above shows.


\end{rem}

In Table \ref{tab:Mazalov} we copy  the row corresponding to $u^{(inf,n,N)}$ from \cite[Table
1]{Seregrina} for the no-information
game with $n$ players and unlimited spins, but without ties (in that version, if
all the scores are above $1$, then the least one is the winner), that is a Drawless Game II.1. If we compare
it with Table \ref{tab:game-2-1}, we can observe how the thresholds in
this variant are noticeable greater than those of the standard Game II.1 for $2$ and
$3$ players, while they decrease towards the same values as $n$ increases.
\begin{table}[h]
  \centering
\caption{Approximate optimal Nash threshold in the Drawless version of Game II.1}%
\label{tab:Mazalov}%
\centering
\begin{tabular}
[c]{cccccccccc}%
$n$ & $2$ & $3$ & $4$ & $5$ & $6$ & $7$ & $8$ & $9$ & $10$\\\hline
{${u}^{\ast}$} & {$0.633$} & {$0.718$} & {$0.767$} & {$0.800$} & {$0.823$} &
{$0.841$} & {$0.856$} & {$0.867$} & {$0.877$}\\\hline
\end{tabular}
\end{table}

\subsection{Game II.2: Zero-sum game}

  We still consider a no-information variant with $0$ payoff if all
the players get score $0$. However, when there is a winner, he is paid
$1/(n-1)$ by each player (so that his total payoff is $1$). In this variant we
will assume that $n>1$.

\begin{lem}
 \label{lem:gamma} For $n\in (1,\infty)$, the equation
 \begin{equation}
   \label{eq:79}
\left(  1+e^{x}(x-1)\right)  ^{n-1}=\dfrac{1}{1+e^{x}(n-1)}, 
\end{equation}
has a single solution $\gamma_n\in (0,1)$.
\end{lem}
  \begin{proof}
Put $f_n(x)=\left(  1+e^x(x-1)\right)^{n-1}$ and $g_n(x)=\dfrac{1}{1+e^x(n-1)}$. Then, we have that $f_n(0)=0$, $f_n(1)=1$, $g_n(0)=1/n>0$ and $g_n(1)=\frac{1}{1+e(n-1)}<1$. Since $f_n(x)$ is increasing and $g_n(x)$ is decreasing in $[0,1]$, the result follows.
\end{proof}

\begin{lem} 
	\label{lem:gamma-increasing} The sequence $\{\gamma_{n}\}_{n>1}$, for $n\in \mathbb{N}$,
	is strictly increasing.
\end{lem}
\begin{proof}
	Write
	$$F(x,n) = f_n(x)-g_n(x)= \left(e^x (x-1)+1\right)^{n-1}-\frac{1}{(n-1) e^x+1}.
	$$
	Now, we have
	\begin{equation*}
		\frac{\partial F}{\partial x} =(n-1) e^x \left(x \left(e^x (x-1)+1\right)^{n-2}+\frac{1}{\left((n-1)
			e^x+1\right)^2}\right)
	\end{equation*}
	which is positive for all real numbers $x\in(0,1)$ and $n\in (1,\infty)$. Thus, the curve $F(x,n)=0$ can be parametrized as $(x(n),n)$ for $x\in(0,1)$ and $n\in (1,\infty)$. Note that, in particular, $\gamma_n=x(n)$ for every integer $n>1$. 

On the other hand, we also have
		$$
	\frac{\partial F}{\partial n} =\frac{e^x}{\left((n-1) e^x+1\right)^2}+\left(e^x (x-1)+1\right)^{n-1} \log
	\left(e^x (x-1)+1\right).
	$$
and, since $F(x,n)=0$ if and only if  $\displaystyle{\left(e^x (x-1)+1\right)^{n-1}=\frac{1}{(n-1) e^x+1}}$, it follows that
		\begin{equation*}
		\frac{\partial F}{\partial n}\Big|_{F(x,n)=0} =\frac{e^x-\frac{\left((n-1) e^x+1\right) \log
				\left((n-1) e^x+1\right)}{n-1}}{\left((n-1)e^x+1\right)^2},
	\end{equation*}
	which is negative for all real numbers $x\in(0,1)$ and $n\in (1,\infty)$. Thus, the implicit function theorem yields that
	\begin{equation*}
		x^{\prime}(n) = \frac{-F_n(x(n),n)}{F_x(x(n),n)} > 0
	\end{equation*}
	which implies that the sequence $(\gamma_n)_{n>1}$ for $n\in \mathbb{N}$ is strictly increasing, as claimed.
\end{proof}
\begin{prop}
  \label{pro:equilibrium-game-2-2} Game II.2 admits a Nash
equilibrium with identical strategies $\gamma_{n}$ for all players. Moreover,
the tying probability is
\begin{equation*}
\widehat{P}_{n}=\left(  1-e^{\gamma_{n}}+e^{\gamma_{n}}{\gamma_{n}}\right)  ^{n},
\end{equation*}
and the winning probability of each player is
\begin{equation*}
P_{n}=\frac{1-{\left(  1-e^{\gamma_{n}}+e^{\gamma_{n}}{\gamma_{n}}\right)
}^{n}}{n}.
\end{equation*}
The expected payoff of each player is obviously $0$.
\end{prop}
\begin{proof}
  Assuming that all rivals use thresholds $\gamma_n$, we will call $\xi_{\gamma_n}^{(i)}$ the random variable that represents the score obtained by the $i$-th player using the threshold $\gamma_n $ and $h(x)$ the probability of winning for a player when stopping with score $x$. We will assume without loss of generality that the reference player is the first and that the rest use a greed threshold $\gamma_n$. It is about proving that, under these conditions, the optimal threshold for the first player is precisely $\gamma_n$.

$\bullet$ If $x \geq\gamma_{n}$ then the first player wins with probability $\prod_{i=2}^{n} P(\xi_{\gamma_n}^{(i)}<x)$ and the probability of a tie is zero. Therefore
\begin{equation}\label{eq:82}
  h(x)=\prod_{i=2}^{n} P(\xi_{\gamma_n}^{(i)}<x)-\frac{1-\prod_{i=2}^{n} P(\xi_{\gamma_n}^{(i)}<x)}{n-1}=
  \frac{n}{n-1} ( 1 + e^{\gamma_n}\,\left(  -1 + x\right)  )^{n-1}-\frac{1}{n-1}.
\end{equation}

$\bullet$ If $x<\gamma_n$ then the first player wins when the remaining players finish with a score of 0 (all exceed 1). {Therefore, he wins with probability $\prod_{i=2}^{n} P(\xi_{\gamma_n}^{(i)}=0)$ and the probability of a tie is zero, so}
\begin{multline*}
  h(x)=\prod_{i=2}^{n} P(\xi_{\gamma_n}^{(i)}=0)- \frac{1-\prod_{i=2}^{n} P(\xi_{\gamma_n}^{(i)}=0)}{n-1}=\\
  \frac{n}{n-1}\prod_{i=2}^{n} P(\xi_{\gamma_n}^{(i)}=0)-\frac{1}{n-1}= \frac{n}{n-1}\frakp{}(\gamma_n)^{n-1}-\frac{1}{n-1},
\end{multline*}
and
\begin{equation*}
  h(0)=-\frac{1-\frakp{}(\gamma_n)^{n-1}}{n-1}.
\end{equation*}

Furthermore, $h(x)$ is non-decreasing in $[0,1]$ so that, by Proposition \ref{OLA}, it is enough to verify that
\begin{equation} \label{eq:winning}
h(\gamma_n)=-\frac{\gamma_n}{n-1}(1-\frakp{}(\gamma_n)^{n-1}) +\int_{\gamma_n}^ 1 h(t) dt.
\end{equation}
From  \eqref{eq:82}, we have
\begin{equation*}
  h(\gamma_n)=\frac{n}{n-1} ( 1 + e^{\gamma_n}\,\left(  -1 + \gamma_n\right)  )^{n-1}-\frac{1}{n-1},
\end{equation*}
and
\begin{equation*}
  \int_{\gamma_n}^ 1 h(t) dt=\frac{-e^{-\gamma_n} \left(\left(e^{\gamma_n} (\gamma_n
        -1)+1\right)^n-1\right)+\gamma_n -1}{n-1}.
\end{equation*}
By the properties of $\gamma_n$ (Lemma \ref{lem:gamma}), we have\begin{equation*}
  \left(  1+e^{\gamma_n}(\gamma_n-1)\right)  ^{n-1}=\dfrac{1}{1+e^{\gamma_n}(n-1)},
\end{equation*}
and equality \eqref{eq:winning} holds.

Finally, the tying probability $\widehat{P}_n$ is obviously the probability of all the players getting a score of $0$ when using $\gamma_n$ as their greed threshold; i.e., $\widehat{P}_n=\frakp{}(\gamma_n)^n$. Consequently, the winning probability of each player is given by
\begin{equation*}
P_{n}=\frac{1-{\frakp{}(\gamma_{n})}^{n}}{n},
\end{equation*}
and Lemma \ref{lem:xi-kappa} gives then
\begin{equation*}
  P_n =\frac{1-{\left(  1-e^{\gamma_{n}}+e^{\gamma_{n}}{\gamma_{n}}\right)
  	}^{n}}{n},
\end{equation*}
as desired.
\end{proof}

Table \ref{tab:threshold-game-2-2} shows $\gamma_{n}$ and the
tying and winning probabilities in Game II.2 for $2\leq n\leq10$.
\begin{table}[h]
  \centering
\caption{Approximate optimal threshold, tie $\widehat{P}_{n}$ and winning $P_{n}$ probabilities
for Game II.2}%
\label{tab:threshold-game-2-2}%
  \centering
\begin{tabular}
[c]{cccccccccc}%
$n$ & $2$ & $3$ & $4$ & $5$ & $6$ & $7$ & $8$ & $9$ & $10$\\\hline
{${\gamma_{n}}$} & {$0.6591$} & {$0.7305$} & {$0.7744$} & {$0.8046$} &
{$0.8268$} & {$0.8440$} & {$0.8576$} & {$0.8689$} & {$0.8783$}\\\hline
{$\widehat{P}_{n} $} & {$0.1163$} & {$0.0855$} & {$0.0680$} & {$0.0566$} & {$0.0486$}
& {$0.0426$} & {$0.0379$} & {$0.0342$} & {$0.0312$}\\\hline
{$P_{n}$} & {$0.4419$} & {$0.3048$} & {$0.2330$} & {$0.1887$} & {$0.1586$} &
{$0.1368$} & {$0.1203$} & {$0.1073$} & {$0.0969$}\\\hline
\end{tabular}
  \end{table}

\begin{rem}
  The Nash equilibrium found in Proposition
\ref{pro:equilibrium-game-2-2} is the only possible one with identical
strategies for all players. It is easy to see that for $n=2$ it is the only
possible equilibrium (symmetric or not), and that the strategies are maximin
(because the game is zero-sum). For $n>2$ we are not sure but it seems
reasonable to think that there are no other (non-symmetric) Nash equilibria,
and neither do there exist any collusive coalitions yielding negative payoffs
to the players out of it. Thus, we conjecture that it is both a strong and a
coalition-proof Nash equilibrium. In Figure 2 we plot the expected payoff of
the first player $\mathbb{P}_{1}(u_{1},u_{2},u_{3})$ in the $3$-player game when they
use respective greed thresholds $u_{1},u_{2},u_{3}$: if the first one sets
$u_{1}=\gamma_{3}$, no combined strategy of the other two yields him a
negative payoff.
\begin{figure}[h]
\centering  
 \label{F2} \includegraphics[width=6in]{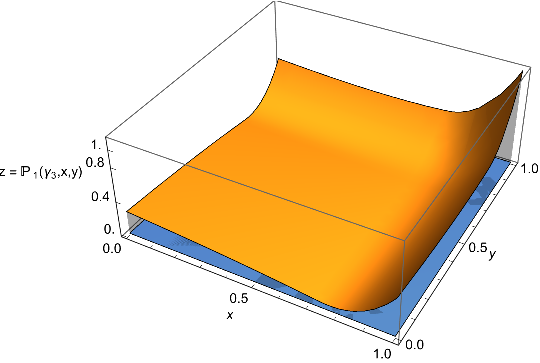}
 \caption{Payoff $z=\mathbb{P}_{1}(\gamma_{3},x,y)$. In blue, $z=0$: notice how $\mathbb{P}_1(\gamma_3,x,y)\geq 0$ everywhere.}%
\end{figure}
\end{rem}

\begin{rem}
  Unlike in Game II.1, the equilibrium strategy is not to maximize
the winning probability but to have a score at least the same as the rest of
the players. As a matter of fact, if all players used the Nash thresholds of
Game II.1, they would have a greater winning probability than with the
$\gamma_{n}$ of Proposition \ref{pro:equilibrium-game-2-2}. However, that
would not provide a Nash equilibrium in Game II.2 because all players would be
better off abandoning that strategy in order to have a greater probability of
success than the others, despite their own probability of winning being less.

\end{rem}

\begin{rem}
  If instead of having the losing players pay the winner, we
consider an external payer, and that if a tie happens, the game is to be
repeated until it is broken, then the Nash equilibrium would be the same and
the expected payoff would be $1/n$. This is because this variant is a
constant-sum game. It is remarkable that the equilibrium strategies are
different when there is no payoff in case of tie, and when, in this case, the
game is repeated until the tie is broken.
\end{rem}

\subsection{ Game II.3: Non-symmetric and constant-sum}

  We now turn our attention to the no-information, non-symmetric
version of Showcase Showdown with $n>1$ players in which one has advantage in
the sense that if there is a tie in scores (i.e., if all players get a score
of $0$) then he is the winner. In this case the payer being external or the
players is irrelevant, as this is a constant-sum game and the only relevant
point is the winning probability.

\pedro{}
\begin{lem} \label{sim}
For all $n>1$, the system of equations
\begin{equation}\label{eq:90}
\left\{
  \begin{aligned}
    \mathrm{a)}\;(1+e^{x}(-1+x))^{n-2}&=
                           \dfrac{e^{y}\,\left( -1+{\left( 1+e^{x}\,\left(
                           -1+y\right) \right) }^{n}\right) +n\,e^{x}}{n\,e^{x}\,
                           \left( 1+e^{y }\,\left( -1+y\right) \right) \,\left(
                           1+e^{x}\,\left( -2+n+x\right) \right) },\\
    \mathrm{b)}\;{(1+e^{x}(-1+x))}^{n-1}&=\frac{e^{-x} \left(n e^x
   \left(e^x (y-1)+1\right)^{n-1}+\left(e^x
      (y-1)+1\right)^n-1\right)}{n y}.
\end{aligned}
\right.
\end{equation}
has a single solution in
$[0,1]\times\lbrack0,1]$ which we will denote from now on
$(\epsilon_{n},\delta_{n})$ with $\delta_{n}>\epsilon_{n}$.
\end{lem}
\pedro{}
\begin{proof}
It can be seen, using implicit differentiation and some computational effort, that the first equation defines $y$ as a decreasing function of $x$, while the second equation defines $y$ as an increasing function of $x$. Then, it is enough to compare the values at $x=0$ and $x=1$ to conclude the proof.
\end{proof}

\begin{figure}[h]
  \centering
 \label{asim} \includegraphics[width=3in]{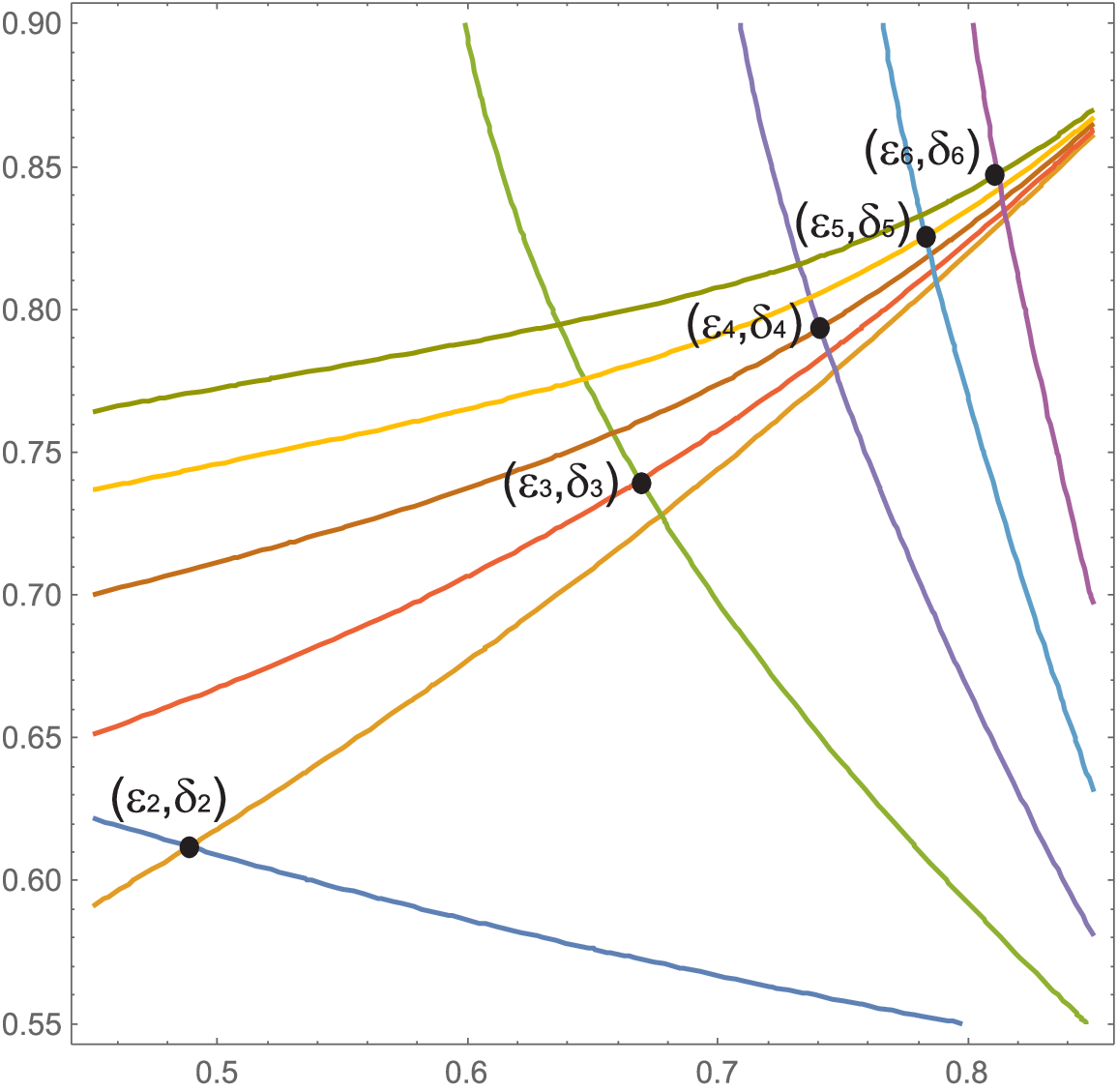}
 \caption{Solutions of Equation \eqref{eq:90} for $n=2,3,\ldots, 6$. The decreasing plots correspond to a) in \eqref{eq:90}, while the increasing ones to b).}%
\end{figure}

  Figure 3 shows the situation of Lemma \ref{sim} for $2\leq
n\leq 6$. The curves represent the points $(x,y)$ satisfying either a) or b) in Lemma \ref{sim}, for $n=2,\ldots,6$, so that their respective intersections are the points $(\epsilon_n,\delta_n)$. The figure clearly suggests that both sequences $\{\epsilon_{n}
\}_{n>1}$ and $\{\delta_{n} \}_{n>1}$ are strictly increasing. Since this fact
is not required in the sequel, and we have not been able to find a direct
proof for every $n>1$, this remains a conjecture.

\begin{con}
 The sequences $\{\epsilon_{n} \}_{n>1}$ and $\{\delta_{n} \}_{n>1}$ of solutions of the system of equations \eqref{eq:90} introduced in Lemma \ref{sim} are strictly increasing.
\end{con}

\begin{prop}
  \label{pro:gameII-3}
  Game II.3 admits a Nash equilibrium with greed threshold
$\delta_{n}$ for the player with advantage and identical greed thresholds
$\epsilon_{n}$ for the remaining players. Moreover, the winning probability
for the player with advantage is
\begin{align}
P_{n}^{A}={\left(  1-e^{\epsilon_{n}}+e^{\epsilon_{n}}\,\epsilon_{n}\right)
}^{n-1}\,\left(  1-e^{\delta_{n}}+e^{\delta_{n}}\,\delta_{n}\right)   &
+\frac{e^{\delta_{n}}\left(  1-{\left(  1+e^{\epsilon_{n}}\,\left(
-1+\delta_{n}\right)  \right)  }^{n}\right)  \,}{e^{\epsilon_{n}}\,n\,}.%
\end{align}

\end{prop}

  \begin{proof}
Without loss of generality we assume that the player with advantage is the
$n$-th one. We will show that the following two statements \textbf{A}) and \textbf{B}) are verified:

\textbf{A})\emph{If the $n$-th player uses the threshold $\delta_{n}$ and the rest except one of them (which we will assume without loss of generality is the first player) use $\epsilon_{n}$ as a threshold, then the optimal threshold of the first is also $\epsilon_{n}$}.

Thus, assume that player $n$ uses threshold $\delta_{n}$ and that
players $2,\ldots,n-1$ use the identical threshold $\epsilon_{n}$. Assuming that the rivals use the thresholds already mentioned, we will call $\xi_{\epsilon_{n}}^{(i)}$, $2\leq i\leq n-1,$ the random variable that represents the score obtained by player $i$ using the threshold $\epsilon_{n}$, $\xi_{\delta_{n}}^{(n)}$ the random variable that represents the score obtained by player $n$ using the threshold $\delta_{n}$ and let $h(x)$ be the probability of winning for the first player when stopping with score $x$.
It is about proving that under these conditions the optimal threshold for the first player is precisely $\epsilon_{n}$.

$\bullet$ For all $x\in\lbrack\epsilon_{n},\delta_{n}],$%
\begin{equation*}
h(x)=P(\xi_{\delta_{n}}^{(n)}=0)\prod_{i=2}^{n-1}P(\xi_{\epsilon_{n}}%
^{(i)}<x)=\,\left(  1+e^{\delta_{n}}\,\left(  -1+\delta_{n}\right)  \right)
\,\left(  1+e^{\epsilon_{n}}\,\left(  -1+x\right)  \right)  ^{n-2}.%
\end{equation*}

  $\bullet$ For $x<\epsilon_{n},$%
\begin{equation*}
h(x)=P(\xi_{\delta_{n}}^{(n)}=0)\prod_{i=2}^{n-1}P(\xi_{\epsilon_{n}}%
^{(i)}=0)=\,\left(  1+e^{\delta_{n}}\,\left(  -1+\delta_{n}\right)  \right)
\,\left(  1+e^{\epsilon_{n}}\,\left(  -1+\epsilon_{n}\right)  \right)  ^{n-2}.%
\end{equation*}

 $\bullet$ For $x>\delta_{n},$%
\begin{equation*}
h(x)=P(\xi_{\delta_{n}}^{(n)}<x)\prod_{i=2}^{n-1}P(\xi_{\epsilon_{n}}%
^{(i)}<x)=\,\left(  1+e^{\delta_{n}}\,\left(  -1+x\right)  \right)  \,\left(
1+e^{\epsilon_{n}}\,\left(  -1+x\right)  \right)  ^{n-2}.%
\end{equation*}

We have that $h(x)$ is continuous and non-decreasing in $[0,1]$ so to be in the conditions of Proposition \ref{OLA} we need to see that

\begin{equation*}
  h(\epsilon_{n})=\int_{\epsilon_{n}}^{1}h(t)dt
\end{equation*}

To see this, we take into account that
\begin{equation*}
h(\epsilon_{n})=\left(  1+e^{\delta_{n}}\,\left(  -1+\delta_{n}\right)
\right)  \,\left(  1+e^{\epsilon_{n}}\,\left(  -1+\epsilon_{n}\right)
\right)  ^{n-2}.%
\end{equation*}
On the other hand,
\begin{multline*}
\int_{\epsilon_{n}}^{1}h(t)dt=\int_{\epsilon_{n}}^{\delta_{n}}%
h(t)dt+\int_{\delta_{n}}^{1}h(t)dt=\\
\int_{\epsilon_{n}}^{\delta_{n}}\left(  1+e^{\delta_{n}}\,\left(
-1+\delta_{n}\right)  \right)  \,\left(  1+e^{\epsilon_{n}}\,\left(
  -1+t\right)  \right)  ^{n-2}dt+\\
\int_{\delta_{n}}^{1}\left(  1+e^{\delta_{n}%
}\,\left(  -1+t\right)  \right)  \,\left(  1+e^{\epsilon_{n}}\,\left(
-1+t\right)  \right)  ^{n-2}dt=\\
\frac{e^{-\epsilon_{n}}\left(  1+e^{\delta_{n}}\,\left(  -1+\delta
_{n}\right)  \right) [ \left(  1+e^{\epsilon_{n}}\,\left(  -1+\delta
_{n}\right)  \right)  ^{n-1}-\left(  1+e^{\epsilon_{n}}\,\left(
-1+\epsilon_{n}\right)  \right)  ^{n-1}]}{n-1}+\\
\frac{e^{-2\epsilon_{n}}[-e^{\delta_{n}}+e^{\epsilon_{n}}n+(1+e^{\epsilon
_{n}}\,\left(  -1+\delta_{n}\right)  )^{n-1}\left(  e^{\delta_{n}}%
-e^{\epsilon_{n}}\left(e^{\delta_{n}}\,\left(  -1+\delta_{n}\right)
(n-1)+n\right)  \right)]}{n(n-1)}.%
\end{multline*}

Taking into account now the definition of $\epsilon_n$ and $\delta_n$ in Lemma \ref{sim}, the equality $h(\epsilon_{n})=\int_{\epsilon_{n}}^{1}h(t)dt$ holds.

\textbf{B}) \emph{If all but the $n$-th player use the threshold $\epsilon_{n}$
then the optimal threshold for the last one is $\delta_{n}$.}

Assuming that all the players (except the last one) use thresholds $\epsilon_{n}$, we will consider the random variable $\xi_{\epsilon_{n}}^{(i)}$, $1\leq i<n,$
which represents the score obtained by the $i$-th player using the threshold $\epsilon_{n}$ and $h(y)$ to be the probability of winning of the $n$-th player by stopping with score $y$.
It is about proving that under these conditions the optimal threshold for the $n$-th player is precisely $\delta_{n}$.

$\bullet$ If $y\geq\epsilon_{n}$, then we have that the $n$-th player wins with probability
\begin{equation}\label{eq:98}
h(y)=\prod_{i=1}^{n-1}P(\xi_{\epsilon_{n}}^{(i)}<y)=\left(  1+e^{\epsilon_{n}
}\,\left(  -1+y\right)  \right)  ^{n-1}.%
\end{equation}

$\bullet$ If $y<\epsilon_{n}$, then the $n$-th player wins if the remaining
players finish with score 0 (all exceed 1)
\begin{equation*}
h(y)=\prod_{i=1}^{n-1}P(\xi_{\epsilon_{n}}^{(i)}=0)=\left(  1+e^{\epsilon_{n}%
}\,\left(  -1+\epsilon_{n}\right)  \right)  ^{n-1}.%
\end{equation*}

We have that $h(y)$ is continuous and non-decreasing in $[0,1]$ so that by Proposition \ref{OLA} it is enough to verify that

\begin{equation*}
h(\delta_n)=\delta_n \frakp{}(\epsilon_{n})^{n-1} +\int_{\delta_n}^{1}h(t)dt.
\end{equation*}

To do this, using \eqref{eq:98} for $y\geq \epsilon_n$, we get:
\begin{multline*}
\int_{y}^{1}h(t)dt+ y \frakp{}(\epsilon_{n})^{n-1}  =\\
-\frac{e^{-\epsilon_{n}}\left(  -1+(1+e^{\epsilon_{n}}\,\left(
-1+y\right)  \right)  ^{n}}{n}+ y \left(  1+e^{\epsilon_{n}}\,\left(
-1+\epsilon_{n}\right)  \right)  ^{n-1}.
\end{multline*}

Performing simple calculations, taking into account the second equation in \eqref{eq:90} in Lemma \ref{sim} with $x=\epsilon_{n}$, it follows
that
\begin{equation*}
h(\delta_n)=\delta_n \frakp{}(\epsilon_{n})^{n-1} +\int_{\delta_n}^{1}h(t)dt.
\end{equation*}
Now, the winning probability of the advantageous player is the sum of the probabilities of the following complementary events:
\begin{itemize}
\item[a)] All the players have a score of $0$. This happens with probability
  \begin{equation*}
    \frakp{}(\epsilon_n)^{n-1}\frakp{}(\delta_n)={\left(  1+e^{\epsilon_{n}}\,\left(  -1+\epsilon_{n}\right)  \right)  }
  ^{n-1}{\left(  1+e^{\delta_{n}}\,\left(  -1+\delta_{n}\right)  \right)  }.
\end{equation*}
\item[b)] The $n$-th player  has a score of $t\in\left[  \delta_{n},1\right]$ and the rest of them (using threshold $\epsilon_{n}$) do not exceed this
score. This happens with probability
\begin{equation*}
  e^{\delta_{n}}\,\left(  1-\delta_{n}\right)  \int_{\delta_{n}}^{1}\frac{(P(\xi_{\epsilon_{n}}<t))^{n-1}}{1-\delta_{n}}dt=e^{\delta_{n}}\frac{1-({1+e^{\epsilon_{n}%
}\,\left(  -1+\delta_{n}\right))}^{n}}{n{e^{\epsilon_{n}}}}.
\end{equation*}
\end{itemize}
Then, $P_{n}^{A}$ is just the sum of the previous probabilities and the result follows.
\end{proof}

Table \ref{tab:nash-game-2-3} shows the Nash thresholds
$\alpha_{n}$ for the ``normal'' players and $\beta_{n}$ for the one with
advantage, and the winning probability $P_{n}^{A}$ of the latter in Game II.3.
Note that the winning probability of the remaining players is just
$P^N_n=\frac{1-P_{n}^{A}}{n-1}$.

\begin{table}[h]
  \centering
\caption{Nash thresholds $(\epsilon_n,\delta_n)$, and winning probability $P^A_n$ of the player with advantage, and
winning probability $P^N_n$ of the other players  in Game II.3}%
\begin{tabular}{cccccccccc}
$n$ & $2$ & $3$ & $4$ & $5$ & $6$ & $7$ & $8$ & $9$ & $10$ \\ \hline
{${\epsilon _{n}}$} & {$0.4887$} & {$0.6687$} & {$0.7408$} & {$0.7832$} & {$%
0.8119$} & {$0.8329$} & {$0.8491$} & {$0.8621$} & {$0.8728$} \\ \hline
{$\delta _{n}$} & {$0.6118$} & {$0.7401$} & {$0.7936$} & {$0.8256$} & {$%
0.8475$} & {$0.8637$} & {$0.8763$} & {$0.8865$} & {$0.8948$} \\ \hline
{$P_{n}^{A}$} & {$0.5366$} & {$0.3720$} & {$0.2879$} & {$0.2357$} & {$0.1998$ } &
{$0.1736$} & {$0.1535$} & {$0.1377$} & {$0.1249$} \\ \hline $P_n^N$ &0.4634 &
0.3140 & 0.2374 & 0.1911&0.1600 & 0.1378& 0.1209 & 
0.1078 & 0.0972\\\hline

\end{tabular}%
\label{tab:nash-game-2-3}
\end{table}

\begin{rem}
 The above Nash equilibrium is the only possible one with
identical thresholds for the disadvantageous players. And, as in the previous
versions, it seems reasonable that there are no more Nash equilibria. This
equilibrium seems also (as in the case of Game II.2) a strong and
coalition-proof Nash equilibrium. There is no possibility of collusion in this game, either.
\end{rem}

\section{Conclusions and future perspectives}

Several versions of the Showcase Showdown game with unlimited
number of spins have been studied, including cases with possibility of draw.
For the sequential $n$-player game, optimal thresholds for each player, and
their winning probabilities have been computed, thus improving on the results in
\cite{Seregrina}, which apply only to the first player. We have also discovered that there is a possibility of coalitions that decrease the probability of winning for certain players, so that the Nash equilibrium does not guarantee the winning probabilities computed.

We have studied three versions of the game in which the players have no information on the score of the others. Despite the underlying game being the same, the Nash equilibria vary considerably, especially for games with few players (see, for instance, cases $n=2$ and $n=3$ in Table \ref{tab:conclusions}). The greedest strategies (those with greater threshold) seem to be those of the player with advantage in the asymmetric game (except when $n=2$, where the greedest threshold happens in the zero-sum game). The equilibrium strategies with lesser threshold seem to happen for all $n$ for the players without advantage, in the asymmetric game. See Table \ref{tab:conclusions} for a summary of these results. We have also established that the equilibrium thresholds are increasing in $n$ in the Non-zero-sum and in the Zero-sum games (we conjecture that this also happens in the Asymmetric one).
\begin{table}
  \label{tab:conclusiones}
  \centering
  \caption{Nash thresholds for the Non-zero-sum game $(\alpha_n)$, the Zero-sum game $(\gamma_{n})$, and the Asymmetric game $(\epsilon_n,\delta_n)$.}
  \begin{tabular}[c]{cccccccccc}\label{tab:conclusions}%
    $n$ & $2$ & $3$ & $4$ & $5$ & $6$ & $7$ & $8$ & $9$ & $10$\\\hline
    {${\alpha_{n}}$} & {$0.5887$} & {$0.6989$} & {$0.7562$} & {$0.7927$} &
                                                                           {$0.8184$} & {$0.8377$} & {$0.8528$} & {$0.8650$} & {$0.8751$}\\\hline
    {${\gamma_{n}}$} & {$0.6591$} & {$0.7305$} & {$0.7744$} & {$0.8046$} &
                                                                           {$0.8268$} & {$0.8440$} & {$0.8576$} & {$0.8689$} & {$0.8783$}\\\hline
    
    {${\epsilon _{n}}$} & {$0.4887$} & {$0.6687$} & {$0.7408$} & {$0.7832$} & {$%
                                                                              0.8119$} & {$0.8329$} & {$0.8491$} & {$0.8621$} & {$0.8728$} \\ \hline
    {$\delta _{n}$} & {$0.6118$} & {$0.7401$} & {$0.7936$} & {$0.8256$} & {$%
                                                                          0.8475$} & {$0.8637$} & {$0.8763$} & {$0.8865$} & {$0.8948$} \\\hline

  \end{tabular}
\end{table}

 Some open questions remain for the case $n>2$: it seems likely
that the symmetric Nash equilibria found in the no-information games are
unique, due to the inherent symmetry in the game; furthermore, in the
constant-sum cases, it seems reasonable to think that the computed strategies
are optimal in the sense that they guarantee at least the expected payoff when
all the players use their Nash thresholds. But we have not been able to tackle
these problems yet.

 We suggest the following future research:

\begin{enumerate}
\item   To consider a maximum value for the scores (accumulated sums) $M>1$. It may be very interesting to study the limits of the optimal thresholds
and winning probabilities in the sequential game of each player when $M$ tends to infinity.

\item   Let the payoffs be a function of the score. This would
intertwine the struggle to win and the aim of getting as great a payoff as
possible.

\item  To choose other underlying random variables different from
$\mathbf{U}[0,1]$. For instance, the exponential distribution looks promising.

\item   A version of the sequential game with more information
might be as follows: instead of each player making all his spins in his
(single) turn, he spins once every turn (or stops), and then the next player
gets the turn iteratively, in the same order (unless he has stopped, in which
case he does not play any more). The game will end when all the players have
stopped. Players who get a score greater than 1 are eliminated. The game stops
when all have been eliminated or stopped, and the winner is the one
with the highest score.
\end{enumerate}

\bibliographystyle{plain}
\bibliography{Bibliography-MM-MC}

\end{document}